\documentclass[12pt]{article}
\usepackage{amsmath}
\usepackage{graphicx}
\usepackage{enumerate}
\usepackage{natbib}
\usepackage{enumitem}
\usepackage{url} 
\RequirePackage{amsthm,amsmath,amsfonts,amssymb}
\RequirePackage[colorlinks,citecolor=blue,urlcolor=blue]{hyperref}
\RequirePackage{graphicx}
\RequirePackage{algorithm}
\RequirePackage{algorithmic}
\usepackage[title]{appendix}

\usepackage{setspace}
\let\Algorithm\algorithm
\renewcommand\algorithm[1][]{\Algorithm[#1]\setstretch{1.2}}

\usepackage{bbm}
\usepackage{xr}

\newtheorem{theorem}{Theorem}

\newtheorem{proposition}{Proposition}
\newtheorem{property}{Property}

\theoremstyle{remark}

\newcommand{\blind}{1}

\begin{document}

\def\spacingset#1{\renewcommand{\baselinestretch}%
{#1}\small\normalsize} \spacingset{1}


\if1\blind
{
  \title{\bf Confidence Intervals for Unobserved Events}
  \author{Amichai Painsky\hspace{.2cm}\\
    Tel Aviv University, Israel}
  \maketitle
} \fi

\if0\blind
{
  \bigskip
  \bigskip
  \bigskip
  \begin{center}
    {\LARGE\bf Confidence Intervals for Unobserved Events}
\end{center}
  \medskip
} \fi

\bigskip

\begin{abstract}
Consider a finite sample from an unknown distribution over a countable alphabet. Unobserved events are alphabet symbols  which do not appear in the sample. Estimating the probabilities of unobserved events is a basic problem in statistics and related fields, which was extensively studied in the context of point estimation. In this work we introduce a novel interval estimation scheme for unobserved events. Our proposed framework applies selective inference, as we construct confidence intervals (CIs) for the desired set of parameters. Interestingly, we show that obtained CIs are dimension-free, as they do not grow with the alphabet size. Further, we show that these CIs are (almost) tight, in the sense that they cannot be further improved without violating the prescribed coverage rate. We demonstrate the performance of our proposed scheme in synthetic and real-world experiments, showing a significant improvement over the alternatives. Finally, we apply our proposed scheme to large alphabet modeling. We introduce a novel simultaneous CI scheme for large alphabet distributions which outperforms currently known methods while maintaining the prescribed coverage rate.  

\end{abstract}

\noindent%
\vfill

\newpage
\spacingset{1.5} 

\section{Introduction}\label{intro}

Consider a probability distribution $p$ over a countable alphabet $\mathcal{X}$. Let $X^n$ be a sample of $n$ independent observations from $p$. Let $N_u(X^n)$ be the number of appearance of the symbol $u$ in the sample.  Unobserved (or unseen) events refer to outcomes which do not appear in the sample, $\mathcal{X}_0(X^n)=\{u\;|\;N_u(X^n)=0\}$. 
In this work we study inference of the unobserved. Specifically, we are interested in simultaneous confidence intervals (SCIs) for the parameters       $\mathcal{M}_p(X^n)=\{p(u)\;|\;N_u(X^n)=0\}$. This problem is of high interest in variety of domains. For example, in $2006$ the New Zealand Ministry of Health reported new cancer cases among local populations \cite{NZ_data}. Specifically, they focused on minorities in different age groups. Their study distinguished between $75$ cancer types. Their report indicated several interesting findings. For example, for Maori females under the age of $30$, there has been a total of $58$ new cancer cases that year. Importantly, out of the $75$ studied cancer types, only $37$ were observed. Does that mean that the remaining $38$ unobserved types of cancer are unlikely to appear in this population? The answer is obviously no, and we would like to infer the likelihood of these unobserved cancer types. Naturally, this example is just a special case of a broader problem, where multiple outcomes are studied and the sample size is limited. 

The classical approach for the studied problem is based on the \textit{rule-of-three}. The rule-of-three (ROT) suggests that for a sample of $n$ observations from a Bernoulli distribution, an approximate CI of level $1-\alpha$ is given by $[0,-\log(\alpha)/n]$ (as shown in Section \ref{Previous_Work}). Extending this result to a multinomial  setup requires a simple multiplicity correction. That is, 
a simultaneous CI (of level $1-\alpha$) for the unobserved satisfies $[0,-\log(\alpha/k)/n]$. As we can see, the obtained CI grows with the alphabet size $k$ and may be too conservative. More importantly, it requires the knowledge of $k$, which is not always available. For example, it is well-known that the number of cancer types is much greater than $75$, despite the report above \cite{SEER_data}. 

In this work we introduce a novel selective inference scheme for unobserved events. That is, we construct CIs only for the events that do not appear in the sample, while refraining from multiplicity correction over the entire alphabet size. To the best of our knowledge, this work is the first to directly address this basic problem. We distinguish between two setups. We first study the case where the alphabet size $k$ is unknown and even unbounded. We obtain a simple closed-form CI that is independent of $k$ and, most importantly, does not grow with it (as opposed to the ROT). Next, we focus on the setup where the alphabet size $k$ is known. Here, we introduce an efficient computational routine which utilizes the alphabet size and further improves our proposed CI. Then, we show that our results are tight. That is, we show that the length of our proposed CI cannot be further reduced (up to a negligible scale). We demonstrate the performance of our proposed scheme in synthetic and real-world experiments, and show it significantly improves upon the alternative.  Finally, we apply our results to large alphabet inference. We introduce a novel simultaneous CI scheme for large alphabet distributions which outperforms currently known methods while maintaining the prescribed coverage rate. 





\section{Previous Work}
\label{Previous_Work}
Consider a set of fixed and unknown parameters  $\theta_1,...\theta_k$. Given a confidence level $1-\alpha$, a simultaneous confidence region for $\{\theta_j\}_{j=1}^k$ is defined as a collection $\{T_j(X^n)\}_{j=1}^k$ such that 

\begin{align}\label{classical}
    P(\cup_j\{\theta_j\notin T_j(X^n)\})\leq \alpha.
\end{align}
In words, the probability that all $\theta_j$ simultaneously reside within their corresponding CI  $T_j(X^n)$ is not smaller than $1-\alpha$. Selective inference generalizes this framework and considers a subset of parameters of interest, selected during the experiment. For example, consider a linear regression problem with feature selection. Naturally, the parameters of interest are those selected by the model. In other words, we would like to infer on a (random) subset of parameters (the \textit{selected parameters}) and not the entire collection.

The problem of inference over selected hypotheses, parameters, or models was  recognized about seven decades ago (see \cite{ben2017concentration} for a detailed discussion). One of the first major contributions to the problem is due to Benjamini and Yekutieli, who considered the problem of constructing CIs for selected parameters \cite{benjamini2005false}. In their work, they showed that conditional coverage, following any selection rule for any set of (unknown) values for the parameters, is impossible to achieve. This means we cannot simply infer on the chosen parameters, given that they were selected. Benjamini and Yekutieli suggested an alternative viewpoint to the problem; instead of controlling the conditional coverage, the obstacle to avoid is that of making a false coverage statement. Specifically, given a selection rule, three outcomes are possible at each experiment; either a covering CI is constructed, a non-covering CI is constructed, or the interval is not constructed at all. Therefore, even though a $1-\alpha$ CI does not offer selective (conditional) coverage, the probability of constructing a non-covering CI is at most $\alpha$, 
\begin{align}\label{SoS}
    P(\cup_j\{\theta_j\notin T_j(X^n), \theta_j\text{ is selected}\})\leq \alpha.
\end{align}
This formulation is also known as \textit{Simultaneous over Selected} (SoS) (Equation (4) in \cite{benjamini2019confidence}). Selective inference was extensively studied over the years. In \cite{benjamini2005false}, Benjamini and Yekulieli relaxed (\ref{SoS}) and defined the \textit{false coverage rate} (FCR) as the expected proportion of parameters not covered by their CIs among the selected parameters. They introduced several controlling procedures for this formulation in different setups. The FCR framework was generalized and applied to a variety of applications. Lee et al. \cite{lee2016exact} and Tibshirani et al. \cite{tibshirani2016exact} constructed confidence interval for parameters selected by the Lasso and by forward step-wise selection, respectively. Berk et al. \cite{berk2013valid} addressed the problem of inference when the model is selected because a pre-specified explanatory variable had the highest statistical significance, which restricts the family over which simultaneous coverage is required. Weinstein and Yekutieli \cite{weinstein2020selective} designed FCR intervals that try to avoid covering zero. See \cite{benjamini2019confidence} for additional references and examples. Benjamini et al. revisited the problem of constructing confidence intervals for seleceted parameters in \cite{benjamini2019confidence}. They defined four controlling formulations, namely SoS (as appears in (\ref{SoS})), FCR, \textit{Simultaneous over all Possible selections} (SoP) and \textit{Conditional over Selected} (CoS),  (See (1)-(4) in \cite{benjamini2019confidence}). They focused their attention to SoS and studied the problem of SoS-controlling CIs for the $r$ largest parameters (of a given collection of parameters). 
A similar framework was also studied by Katsevich and Ramdas \cite{katsevich2020simultaneous}, who addressed simultaneous selective inference in testing under SoP. Specifically, They considered making selective inference statements on many selection rules, guaranteeing these statements hold simultaneously with high probability. 

In this work we study interval estimation of the unseen. This task may be viewed as a selective inference problem, as we construct CIs only for the parameters of the events that are missing from the sample. Since the selected parameters are data-dependent we focus our attention to the SoS framework (\ref{SoS}). In this sense, our work is an application of SoS-controlling CIs for this important problem.

Estimation of the unseen has been extensively studied over the years. Interestingly, most work focus on point estimation, in a variety of setups and applications. Perhaps the first major contribution to this problem dates back to Laplace in the $18^{th}$ century \cite{laplace2012pierre}. In his work, Laplace studied the \textit{sunrise problem};  given that the sun raised every morning until today, what is the probability that it will rise tomorrow? Laplace addressed the problem of unobserved events by adding a single count to all $k$ events in the alphabet (including the unobserved). Then, the desired estimate is simply the empirical frequency,   $1/(n+k)$. This scheme is also known as the  \textit{rule of succession}. The Laplace estimator was later generalized to a family of \textit{add-constant} estimators. An add-$c$ estimator assigns to a symbol that appeared $t$ times a probability proportional to $t+c$, where $c$ is a pre-defined constant.  Add-constant estimators hold many desirable properties, mostly in terms of their simplicity and interpretability \cite{orlitsky2003always}. 
However, when the alphabet size $k$ is large compared to the sample size $n$, add-constant estimators perform quite poorly \cite{orlitsky2003always}. Furthermore, add-$c$ estimators require the knowledge of the alphabet size $k$, which is not always available. Additional caveats of add-$c$ estimators are discussed in  \cite{gale1994s}.

Many years after Laplace, a major milestone was established in the work of Good and Turing \cite{good1953population},  while trying to break the Enigma Cipher during World War II \cite{orlitsky2003always}. The Good-Turing (GT) framework suggests that unobserved events shall be assigned a probability proportional to the number of events with a single appearance in the sample. This approach introduced a significant improvement compared to known estimators at the time. Furthermore, its promising performance and practical appeal have led many researchers to study and generalize these ideas. 

Unseen estimation is highly related to the \textit{missing mass} problem; the total probability of symbols which do not appear in the sample. Estimating the missing mass is a basic problem in statistics and related fields (see \cite{battiston2020consistent} and references therein). Further, it  corresponds to an important prediction  task. Namely, the problem of estimating the likelihood of encountering a future event which does not appear in the sample. Here too, the most popular approach is the GT estimator   (which is, again, proportional to the number of events with a single appearance in the sample). A variety of results were introduced over the years, focusing on the  properties of the missing mass the GT estimator. This includes, for example, asymptotic normality and large deviations \cite{gao2013moderate}, admissibility and concentration properties 
\cite{ben2017concentration}, expectation, consistency and convergence rates \cite{mcallester2000convergence,drukh2005concentration,mossel2019impossibility,painsky2021refined,painsky2022convergence,painsky2022data}, Bayesian estimation schemes  
\cite{lijoi2007bayesian,favaro2016rediscovery,favaro2012new}, and the
estimation of the missing mass in the context of feature models under minimax risk \cite{ayed2019good}. 

As mentioned above, there are many results on point and interval estimation of the missing mass. Yet, these results only apply to the total mass of the unseen. Notice that this problem is fundamentally different than ours. Specifically, we are interested in inferring the probability of each unobserved outcome (as required, for example, in the Maori cancer study), and not their sum. To the best of our knowledge, inference of the unobserved is currently considered only in the simple binomial case ($k=2$). The \textit{rule-of-three} (ROT) suggests that for a sample of $n$ identical samples from a Bernoulli distribution with a parameter $\theta$, the edge of the CI is given by $P(X^n=0) = \alpha$ which leads to $(1-\theta)^n = \alpha$. This implies  $n\log(1-\theta)\approx -n\theta=\log(\alpha)$, where the approximation follows from $\log(1-\theta)\approx-\theta$, for $\theta$ close to zero. Plugging $\alpha=0.05$ results in a one-sided CI of approximately $3/n$ (henceforth, rule-of-three). As mentioned in Section \ref{intro}, the ROT may be generalized to the control all the events that do not appear in the sample by applying a simple Bonferroni correction, leading to a CI of $[0,-\log(\alpha/k)/n]$. The ROT was further extended in different setups. For example, the Vysochanskij-Petunin inequality \cite{vysochanskij1980justification} shows that the ROT holds for unimodal distributions with finite variance, beyond just the binomial distribution. To the best of our knowledge, our contribution is the first to directly address SCIs for unobserved events in the multinomial setup. 

\section{Problem Statement}
\label{defs}
Denote the collection of \textit{missing probabilities} as 
\begin{align}
\mathcal{M}_p(X^n)=\{p(u)\;|\; N_u(X^n)=0\}.
\end{align}
We are interested in simultaneous one-sided CIs for the parameters in $\mathcal{M}_p(X^n)$. This corresponds to constructing a CI only for the greatest element in the set. Hence, our statistic of interest follows 
\begin{align}\label{M_max}
    M_{max}(X^n)=\max_{u \in \mathcal{X}} \mathcal{M}_p(X^n)=\max_{u \in \mathcal{X}}\{p(u)\mathbbm{1}(N_u(X^n)=0)\},
\end{align}
where $\mathbbm{1}(\cdot)$ is the indicator function. Notice that $M_{max}(X^n)$ depends on the (unknown) probability $p$, which is omitted from the syntax for brevity. Our goal is to construct a one-sided CI for $M_{max}(X^n)$, in a confidence level of $1-\alpha$. Specifically, we are interested in $T(X^n)$ such that
$P(M_{max}(X^n)\geq T(X^n))\leq \alpha.$
Unfortunately, the maximum operator is an involved,  non-smooth functional. Therefore, we begin our analysis by representing $M_{max}(X^n)$ as the limit of an $r$-norm. Specifically, given a sample $X^n$ and a fixed parameter $r\geq 1$, we define   
\begin{align} \label{M_r}
    M_r(X^n)\triangleq\sum_{u\in\mathcal{X}}p^r(u)\mathbbm{1}(N_u(X^n)=0).
\end{align}
The $r$-norm of the missing probabilities follows
\begin{align}
    ||\{p^r(u)\mathbbm{1}(N_u(X^n)=0)\}_{u 
    \in \mathcal{X}}||_r\triangleq (M_r(X^n))^{1/r}.
\end{align}
Consequently, we have that 
\begin{align}\label{r-norm lim}
\lim_{r\rightarrow \infty} (M_r(X^n))^{1/r}=\max_{u \in \mathcal{X}}p(u)\mathbbm{1}(N_u(X^n)=0)\triangleq M_{max}(X^n)
\end{align}
and 
$(M_t(X^n))^{1/t}\leq (M_{r}(X^n))^{1/r}$ 
for any $1\leq r\leq t$ \cite{maddox1988elements}. This means that $M_{max}(X^n)\leq (M_r(X^n))^{1/r}$ for every $r\geq 1$. 
Therefore, a $(1-\alpha)$-level confidence interval for $(M_r(X^n))^{1/r}$ is also an  $(1-\alpha)$-level confidence interval for $M_{max}(X^n)$,
\begin{align}\label{CI r-norm}
    P\left(M_{max}(X^n)\geq T(X^n)\right)\leq P\left((M_r(X^n))^{1/r}\geq T(X^n)\right)\leq \alpha.
\end{align}
Denote the expected value of $M_r(X^n)$ as
$E_{r,n}(p)\triangleq\mathbb{E}_{X^n \sim p}(M_r(X^n)).$
Define the worst-case (supremum) of $E_{r,n}(p)$ over $\mathcal{P}$ as

\begin{equation}\label{E_r(P)}
E_{r,n}(\mathcal{P}) \triangleq \sup_{p \in \mathcal{P}} E_{r,n}(p).
\end{equation}  
In this work we focus on two sets of probability distributions $\mathcal{P}$. Let $\Delta_k$ be the set of all distributions of an alphabet size $k$, while  $\Delta$ be the set of all distributions over any countable alphabet $\mathcal{X}$ (that is, $k\rightarrow \infty$). Markov's inequality suggests that for every $\lambda>0$,
\begin{align}\label{Markov}
    P(M_r(X^n) \geq \lambda)\leq \frac{E_{r,n}(p)}{\lambda}\leq \frac{E_{r,n}(\mathcal{P})}{\lambda},
\end{align}
where the second inequality follows from (\ref{E_r(P)}). Setting  $\alpha=E_{r,n}(\mathcal{P})/\lambda$, we have that  
\begin{align}\label{CI}
    &P\left(M_r(X^n)\geq \frac{E_{r,n}(\mathcal{P})}{\alpha}\right)=P\left(\left(M_r(X^n)\right)^{1/r}\geq \left(\frac{E_{r,n}(\mathcal{P})}{\alpha}\right)^{1/r}\right)\leq \alpha.
\end{align}
Plugging $\mathcal{P}=\Delta$ (alternatively, $\mathcal{P}=\Delta_k$), we obtain a one-sided confidence interval for $\left(M_r(X^n)\right)^{1/r}$ (and henceforth, $M_{max}(X^n)$) which holds for every $p \in \Delta$ (alternatively, $\Delta_k$). Notice that the obtained confidence interval is independent of the sample $X^n$, similarly to the ROT. This makes it a  robust non-random scheme, that generalizes the ROT for the multinomial setup. Further, notice that for $r=1$, (\ref{CI}) is a CI for the missing mass. In that sense, our proposed framework generalizes the missing mass problem, and introduces CIs for any $r$-norm of the missing probabilities, $M_r(X^n)$. Interestingly, point estimation of $M_r(X^n)$ was recently studied by Chandra and Thangaraj in quite a different context \cite{chandra2021estimation}. In their work, they showed that for $r \in (1,\infty)/\mathbb{N}$,
\begin{align}\label{chandra}
    \min_{\hat{M}_r (X^n)}\max_{p\in \Delta} \mathbb{E}(M_r(X^n)-\hat{M}_r(X^n))^2\leq O\left(\frac{1}{n^{2(r-1)}}\right)
\end{align}
where $\mathbb{N}$ is the set of natural numbers and $O(\cdot)$ is the standard big O notation \cite{bachmann1894analytische}. Similarly, for $n\geq2r$ and $r\in \mathbb{N}$, they attained a bound of $O\left({1}/{n^{2r-1}}\right)$. 
We discuss these bounds and compare them to our results later in Section \ref{unbounded_alphabet}. 

Taking a closer look at our derivation steps, one may wonder if the obtained data-independent CI is too conservative. Later in Section \ref{tightness}, we show that the proposed CI is indeed tight, in the sense that there exists a distribution $p$ for which (\ref{CI}) is attained with equality. This is not the typical notion of tightness in inference literature, where one would expect a CI to be tight for every $p$. Yet, this approach is quite common in the study of the unobserved. For example, notice that the popular ROT is also only tight in this sense, where only $\theta=-\log(\alpha)/n$ attains it equality. Additional examples in missing mass literature are discussed in \cite{rajaraman2017minimax,ben2017concentration,acharya2018improved,painsky2022generalized}. Let us proceed with our analysis and study $E_{r,n}(\mathcal{P})$ for both bounded ($\mathcal{P}=\Delta_k$) and unbounded ($\mathcal{P}=\Delta$) alphabets.  

\section{Unbounded Alphabet Size}
\label{unbounded_alphabet}

We begin our analysis with the unbounded alphabet setup. First, the expected value of $M_r(X^n)$ satisfies
\begin{align}\label{E(p) explicit}
    E_{r,n}(p)=\sum_{u\in \mathcal{X}} p^r(u)\mathbb{E}_{X^n \sim p} \left(\mathbbm{1}(N_u(X^n)=0) \right)=\sum_{u\in \mathcal{X}} p^r(u)(1-p(u))^n.
\end{align}
We would now like to bound (\ref{E(p) explicit}) from above, for every possible $p\in \Delta$. For this purpose, we introduce the following proposition.
\begin{proposition} \label{prop1}
Let $p$ be a distribution over a countable alphabet $\mathcal{X}$. Let $\phi:[0,1] \rightarrow \mathbb{R}$. Then, 
    $\sum_{u} p(u)\phi(p(u))\leq \max_{q \in [0,1]}\phi(q).$
Further, equality is achieved for a uniform $p$. 
\end{proposition}
\begin{proof}
Let $Y \sim p$ and define a random variable $T(u)$, such that $T(u)=\phi(p(u))$. Then,
        $$\mathbb{E}(T(Y))=\sum_{u,\in \mathcal{X}} p(u)\phi(p(u))\leq \max_{q \in [0,1]}\phi(q),$$
        where in the last inequality, the expectation of a random variable is bounded from above by its maximal value.
        Notice that equality is achieved if all $p(u)$'s are equal.
     
\end{proof}
\noindent Applying Proposition \ref{prop1} to (\ref{E(p) explicit}) we obtain
\begin{align}\label{basic_unbounded}
    E_{r,n}(p)=\sum_{u\in \mathcal{X}} p^r(u)(1-p(u))^n\leq \max_{q\in[0,1]} q^{r-1}(1-q)^n=(q_{r,n}^*)^{r-1}(1-q_{r,n}^*)^n,
\end{align}
where $q_{r,n}^*=(r-1)/(r-1+n)$.
Further, equality is obtained for $p(u)=q_{r,n}^*$, which implies an alphabet size of $k=(r-1+n)/(r-1)$.
To conclude, for a given $r\geq1$ and an unbounded alphabet size, we have that 
\begin{align}\label{temp}
    E_{r,n}(\Delta)=(q_{r,n}^*)^{r-1}(1-q_{r,n}^*)^n,
\end{align}
which further implies that $E_{r,n}(\Delta)=O(1/n^{r-1})$. 
In addition, the distribution which attains the above is a uniform distribution over an alphabet size $k=(r-1+n)/(r-1)$.  Finally, a one-sided confidence interval for $(M_r(X^n))^{1/r}$ is necessarily a one-sided confidence interval for $M_{max}(X^n)$, for every $r\geq 1$. This leads to the following theorem. 

\begin{theorem}\label{theorem_analytical}
Let $p$ be a probability distribution over a countable alphabet $\mathcal{X}$. Let $X^n$ be $n$ independent samples from $p$. Let $M_{max}(X^n)$ be the maximum over the set of missing probabilities, as defined in (\ref{M_max}). Then, the following holds,   
\begin{align}\label{CI_unbounded}
    P\left(M_{max}(X^n)\geq \min_{r\geq 1}\left((q_{r,n}^*)^{r-1}(1-q_{r,n}^*)^n/\alpha  \right)^{1/r}\right)\leq\alpha 
\end{align}
where $q_{r,n}^*=(r-1)/(r-1+n)$. 
\end{theorem}

Notice that the obtained one-sided CI (\ref{CI_unbounded}) is independent of the alphabet size $k$. This means that our proposed CI is constant, and does not grow with $k$, as opposed to the Bonferroni-corrected ROT (see Section \ref{Previous_Work}). As we further examine our results, we observe that for every fixed $r\geq 1$, the $r$-norm of the missing probabilities $M_r(X^n)$ satisfies 
\begin{align}\label{ours}
    P\left(M_{r}(X^n)\geq E_{r,n}(\Delta)/\alpha  \right)\leq\alpha 
\end{align}
where $E_{r,n}(\Delta)=O(1/n^{r-1})$. Let us compare this result to \cite{chandra2021estimation}. Applying Markov's inequality to (\ref{chandra})  we obtain 
\begin{align}
    P\left(|M_{r}(X^n)-\hat{M}_r(X^n)|\geq \lambda\right)\leq\max_{p\in\Delta} \mathbb{E}\left(M_{r}(X^n)-\hat{M}_r(X^n)\right)^2/\lambda^2.
\end{align}
Interestingly, for $r>1$ this leads to a one-sided CI of length $O(1/n^{r-1})$, similarly to (\ref{ours}). However, we emphasize that (\ref{chandra}) is obtained by a data-dependent estimator of $M_r(X^n)$, which also depends on $r$. This means that the choice of $r$ which minimizes the CI for $M_{max}(X^n)$ (as in (\ref{CI_unbounded})) also depends on the sample and is henceforth invalid. However, this analysis emphasizes the tightness of our bound (\ref{basic_unbounded}) and its resulting CI for $M_{max}(X^n)$, even if we compare it to a data-dependent scheme. 

\section{Bounded Alphabet Size}

Let us now study the case where the alphabet size is bounded from above. This is a typical setup, for example, in experimental studies where the number of outcomes is known a priori to the experiment. As discussed in Section \ref{defs}, our proposed CI depends $E_{r,n}(\mathcal{P})$, where $\mathcal{P}=\Delta_k$ in this setup. Therefore, our goal is to  maximize $E_{r,n}(p)$ over $p \in \Delta_k$,
\begin{align} \label{bounded_objective}
E_{r,n}(\Delta_k)=\max_{p\in \Delta_k} \sum_{u\in\mathcal{X}}p^r(u)(1-p(u))^n.
\end{align}
Unfortunately, this optimization problem does not hold a closed form solution. However, we show that it may be efficiently evaluated from  simple optimization considerations. We begin our analysis with the following property.
\begin{property}\label{property1}
Let $E_{r,n}(p)=\sum_{u \in \mathcal{X}} p^r(u)(1-p(u))^n$. Let $$t^*=\frac{r}{r+n},\;\;t_{1,2}=t^*\pm \frac{1}{r+n}\sqrt{\frac{rn}{r+n-1}}.$$
Assume $0\leq t_1\leq t^*\leq t_2\leq 1$. 
For $r\geq 1$, the summand, $h(t)=t^r(1-t)^n$ satisfies the following:
\begin{enumerate}
    \item $h(t)$ has a local maximum at $t^*$
    \item $h(t)$ is concave in $t$, for $t_1 \leq t \leq t_2$.
    \item $h(t)$ is convex in $t$, for $0\leq t\leq t_1$ and $t_2\leq t \leq 1$ .
    
\end{enumerate}
\end{property}
The proof of the above directly follows from the derivatives of the summand, $h(t)=t^r(1-t)^n$, and is located in Appendix A. Property \ref{property1} shows that the function, $p^r(u)(1-p(u))^n$, consists of three separate regions, characterized by their concavity and convexity. This allows us to characterize the maximum of our objective.

\begin{theorem}\label{theorem2}
Let $p^* \in \Delta_k$ be the maximizer of $E_{r,n}(p)=\sum_u p^r(u)(1-p(u))^n$ over $\Delta_k$. Then, for $r\geq 1$ the following holds.
\begin{enumerate}
\item $p^*(u)=p^*(v)$ for every  $p^*(u),p^*(v)\in \left[t_1,t_2\right]$.
\item There exists at most a single $p^*(u)$ such that $p^*(u)\in \left(0,t_1\right)$

\item There exists at most a single $p^*(u)$ such that $p^*(u)\in \left(t_2,1\right]$

\end{enumerate}
\end{theorem}
In words, all $p^*(u)$ that are located in the concave region are identical, and there exists at most a single $p^*(u)$ in the interior of the convex regions. These properties are a direct consequence of the convexity/concavity regions of the summand. The detailed proof is located in Appendix B. Proposition \ref{theorem2} shows that the maximizer of $E_{r,n}(p)$ over $\Delta_k$ depends on not more than four free parameters. Surprisingly, this results holds for every $k$. In other words, we may numerically evaluate $E_{r,n}(\Delta_k)$ by considering only four free parameters, for every given $k$. This allows us to numerically evaluate the CI in a relatively small computational cost, even when the dimension of the problem increases, for every examined $r\geq 1$, and choose the value of $r$ which minimizes the CI (similarly to (\Ref{CI_unbounded})).

\section{Tightness Analysis}\label{tightness}
The derivation of the proposed CI utilizes several relaxations and inequalities, such as the $r$-norm (\ref{CI r-norm}) and Markov inequality (\ref{Markov}). Therefore, it is of a reasonable concern that the obtained CI is over pessimistic. Here, we show that this it not the case. Specifically, we show that there exists a  distribution $p\in \Delta$ for which the proposed CI is (almost) tight.  

As we revisit our analysis in the unbounded alphabet setup (\ref{basic_unbounded}), we observe that  $E_{r,n}(p)\leq (q^*_{r,n})^{r-1}(1-q^*_{r,n})^n$, where equality holds if $p(u)=q^*_{r,n}$. In words, $E_{r,n}(p)$ attains its maximum for a uniform distribution over an alphabet of size $k^*=1/q^*_{r,n}$. This means that in practice, even if the alphabet size $k$ is known to be greater than $k^*$, the worst-case distribution which attains $E_{r,n}(p)$ with equality is a uniform distribution with $p(u)=q^*_{r,n}$ for $k^*$ symbols, and $p(u)=0$ for the remaining alphabet. Interestingly, this type of distributions also attain Markov inequality with equality. Specifically, let $\mathcal{U}_k$ be the set of uniform distributions over an alphabet size $m\leq k$.  Then, for every $p_m\in\mathcal{U}_k$, we have $M_{max}(X^n)\in\{0,1/m\}$ and $P(M_{max}(X^n)\geq 1/m)=m{\mathbb{E}_{X^n\sim p}M_{max}(X^n)}$. This motivates exploring $\mathcal{U}_k$ as a set of distributions for which our proposed CI may be tight. 

As mentioned in Section \ref{defs}, deriving an exact CI for $M_{max}(X^n)$, even when the underlying distribution $p$ is known, is not an easy task. However, we now show it is possible in several special cases, such as $p_m\in \mathcal{U}_k$. We begin with the following proposition.   

\begin{proposition} \label{prop2}
    Let $X^n$ be a sample of $n$ independent observations from $p_m \in \mathcal{U}_k$. Then, 
    \begin{align}\label{worst-case}
   P\left(M_{max}(X^n) \geq \frac{1}{m}\right)=1-\frac{m!S(n,m)}{m^n}
   \end{align}
where $S(n,m)=\frac{1}{m!}\sum_{j=0}^m(-1)^j\binom{m}{j}(m-j)^n$ is \textit{Stirling number of the second kind}.
\end{proposition}
   
The proof of Proposition \ref{prop2} utilizes simple combinatorial properties. Specifically, given that $p_m$ is a uniform distribution over an alphabet size $m$, we have that $M_{max}(X^n)=1/m$ if and only if there exists at least one symbol that do not appear in the sample, where all symbols are equiprobable. A detailed proof is provided in  Appendix C.

Now, define $m_\alpha$ as the largest value of $m$ for which $1-\frac{m!S(n,m)}{m^n}\leq \alpha$. Then, $1/m_\alpha$ is the $\alpha$-quantile of $M_{max}(X^n)$. This means that for $X^n \sim p_{m_\alpha}$, we cannot set any constant $c<1/m_\alpha$ such that $P(M_{max}(X^n)\geq c)\leq \alpha$. In other words, $p_{m_\alpha}$ requires a CI larger than $1/m_{\alpha}$ in order to control $M_{max}(X^n)$ in a confidence level of at least $1-\alpha$. The implications of this result are fairly simple. In order to control every $p\in \Delta$ in a confidence level of $1-\alpha$, a (constant) confidence interval of size of at least $1/m_\alpha$ is inevitable. In other words, $p_{m_\alpha}$ is the distribution which requires the tightest CI (among all the distributions in $\mathcal{U}_k$), and it is therefore the most challenging to control. We denote it as the \textit{worst-case} distribution.  In the following section we compare our proposed CIs with $1/m_\alpha$ and show that the difference is practically negligible.   

\section{Experiments}\label{experiments}
We now illustrate the performance of our proposed CIs in synthetic and real-world experiments. First, we study six example distributions, which are common benchmarks for probability estimation and related problems \cite{orlitsky2015competitive}. The Zipf's law distribution is a typical benchmark in large alphabet probability estimation; it is a commonly used heavy-tailed distribution, mostly for modeling natural (real-world) quantities in physical and social sciences, linguistics, economics and others fields \cite{saichev2009theory}. The Zipf's law distribution follows $p(u;s,k)={u^{-s}}/{\sum_{v=1}^k v^{-s}}$ where $k$ is the alphabet size and $s$ is a skewness parameter.  Additional examples of commonly used heavy-tailed distributions are the geometric distribution, $p(u;\alpha)=(1-\alpha)^{u-1}\alpha$, the negative-binomial distribution (specifically, see \cite{efron1976estimating}), $p(u;l,r)= \binom{u+l-1}{u} r^u(1-r)^l$ and the beta-binomial distribution $p(u;k,\alpha,\beta)= \binom{k}{u}{B(u+\alpha,k-u+\beta)}/{B(\alpha,\beta)}$. Notice that the support of the geometric and the negative binomial distributions is infinite. Therefore, for the purpose of our experiments, we truncate them to an alphabet size $k$ and normalize accordingly. Additional example distributions are the uniform, $p(u)=1/k$, and the worst-case distribution, which is simply a uniform distribution over an alphabet size $1/m_\alpha$, as discussed in Section \ref{tightness}. 

In each experiment we draw $n=1000$ samples, and compare the lengths of different CIs for an increasing alphabet size. Figure \ref{simulated_exp} illustrates the results we achieve. The red curve on top 
corresponds to the Bonferroni-corrected ROT, as discussed in Section \ref{Previous_Work}. As expected, it grows logarithmically with the alphabet size $k$. The blue curve below it is our proposed CI, for a known alphabet size $k$, while the blue dashed curve corresponds to the unbounded alphabet size. As we can see, the bounded $k$ curve is of similar length the ROT CI, for smaller values of $k$. However, as the alphabet size increases, it converges to the unobounded $k$ performance, as expected.  It is also evident that while the ROT CI grows with $k$, our proposed schemes are fixed, and demonstrate significantly shorter confidence intervals, while maintaining the desired coverage rate. As we examine the value of $r$-norm which minimizes our CI, we observe that it increases with $k$ (for the bounded $k$ scheme) and converges to approximately $r=10$. Finally, the black curve at the bottom is an Oracle CI, who knows the underlying distribution $p$. Specifically, the Oracle CI is simply the $\alpha$-quantile of $M_{max}(X^n)$ in the case where $p$ is known. This serves us as a lower bound, for the best we can achieve in each experiment. We focus our attention to the worst-case distribution, which we study in detail in Section \ref{tightness}. As we can see, the Oracle CI is almost identical to our proposed scheme in this setup. This means that our CIs are universally tight, in the sense that there exists a  distribution for which they we cannot be shorten. 

\begin{figure}[ht]
\centering
\includegraphics[width =0.8\textwidth,bb=10 70 750 540,clip]{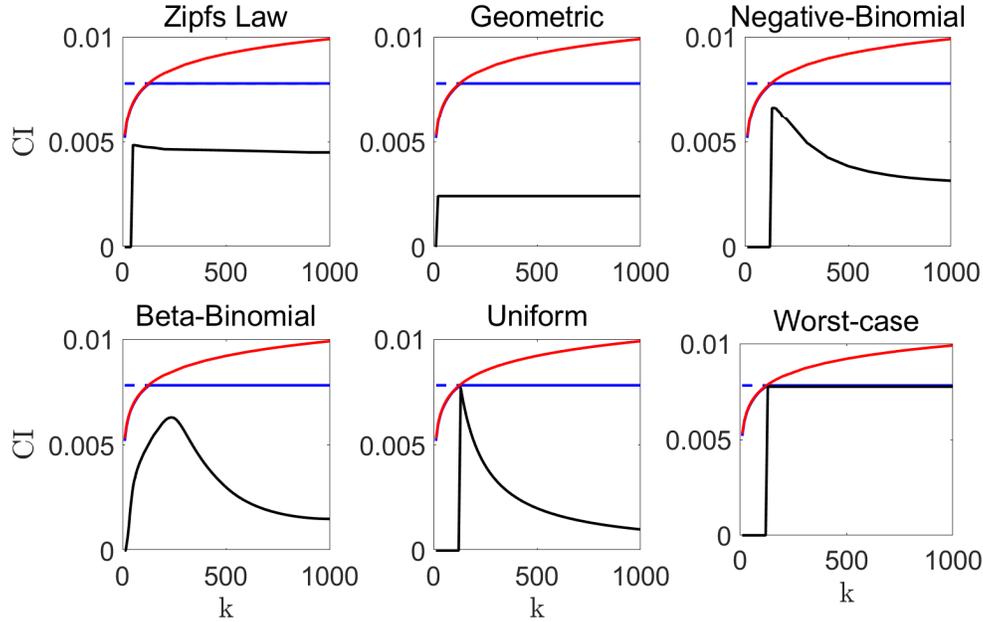}
\caption{confidence intervals lengths in six synthetic experiments. We use the following parameters: Zipf's Law: $s=1.01$, Geometric: $
\alpha=0.4$, Negative-Binomial: $l=1, r=0.003$, Beta-Binomial: $\alpha=\beta=2$. The red curve on top is the rule-of-three CI, the blue curve below it is our proposed CI for a known alphabet size $k$. The blue dashed curve is our proposed method for unbounded $k$ and the black curve at the bottom is an Oracle CI, who knows the underlying distribution $p$.}
\label{simulated_exp}
\end{figure}

Next, we turn to real-world experiments. Here, we follow \cite{orlitsky2016optimal} and study three application domains. Notice that in these real-world settings, the true underlying probability is unknown. Hence, the missing probabilities refers to the frequency of symbols, in the full data-set, that do not appear in the sample. We begin with a corpus linguistic experiment. For this purpose we study a collection of word frequencies in English. 
Specifically, we consider a list of word frequencies, collected from open source subtitles \cite{subtitles1,subtitles2}. This list describes the frequency each word appears in the text, based on hundreds of millions of samples. We randomly sample $n$ words (with replacement) from the list, and construct a CI for the missing probabilities. The left chart of Figure \ref{real_world_exp} demonstrates the CI of our proposed scheme, compared to the Bonferroni-corrected ROT. Notice we focus on the unbounded $k$ scheme, as it is more robust and may better describe the alphabet size in this setup (all the words in the English language). Next, we focus on a biota analysis. Gao et al. \cite{gao2007molecular} considered the forearm skin biota of six subjects. They identified a total of $1{,}221$ clones consisting of $182$ different species-level operational taxonomic units (SLOTUs). As above, we sample $n$ out of the $1{,}221$ clones with replacement, and construct CI for the missing probabilities of the distinct SLOTUs found. The middle chart of  Figure \ref{real_world_exp} demonstrates the results we achieve. Finally, we study census data. The right chart of Figure  \ref{real_world_exp} considers the $2000$ United States Census \cite{us2014frequently}, which lists the frequency of the top $1000$ most common last names in the United States. Here too, we sample $n$ names and construct corresponding CIs. Similarly to the synthetic experiments, our proposed scheme demonstrates shorter CIs than the ROT in the three examined setups, where the difference is typically more evident in the small $n$ regimes. 

\begin{figure}[ht]
\centering
\includegraphics[width =0.8\textwidth,bb=10 140 745 470,clip]{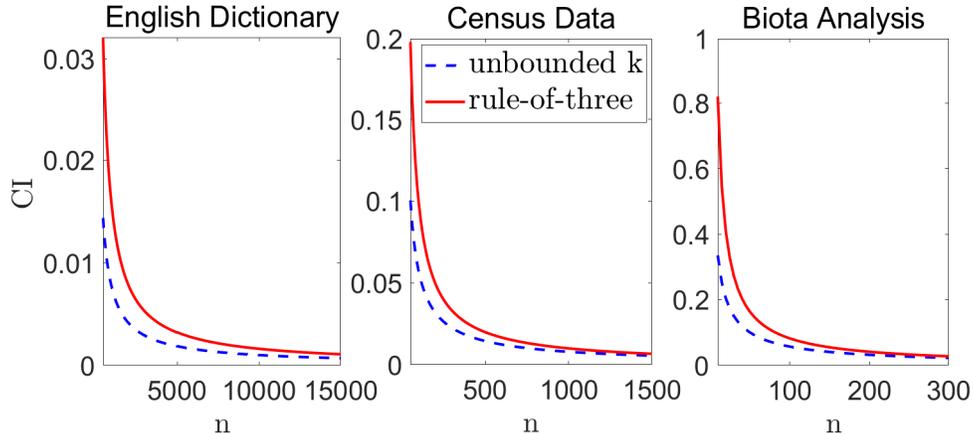}
\caption{confidence intervals lengths in three real-world experiments.}
\label{real_world_exp}
\end{figure}

\section{Application to Large Alphabet Inference}\label{multinomial_intro}

Inference and estimation of unseen events is a corner-stone of large alphabet probability modeling. Here, the goal is to address not only the unobserved events, but the entire (very) large collection of possible outcomes. Specifically, the large alphabet regime considers multinomial distributions in cases where $k$ is much larger than $n$ (or at least comparable to it). This problem too is highly important to data-driven science and engineering. Its applications span a variety of disciplines  including  information retrieval \cite{song1999general}, spelling correction \cite{church1991probability}, word-sense disambiguation \cite{gale1992method}, language modeling  \cite{chen1999empirical}, learning theory \cite{makur2020estimation} and many others. In this section we demonstrate the favorable properties of our proposed CI, as we apply it to large alphabet inference.  

There exists a large body of work on point estimation of large alphabet distributions. The GT probability estimator (based on the GT scheme described above) is perhaps the most popular estimator for this important task. While the GT estimator performs well in general, it is known to be sub-optimal for outcomes that frequently appear. Consequently, several modifications have been proposed, including the Jelinek-Mercer, Katz, Witten-Bell and Kneser-Ney estimators \cite{chen1999empirical}. In language modeling for example, GT is usually used to estimate the probability of infrequent words, whereas the probability of frequent words is estimated by their empirical frequency. Different properties of the GT probability estimator were extensively studied over the years \cite{mcallester2000convergence,drukh2005concentration,orlitsky2015competitive}. 
Despite this broad body of work, large alphabet inference has not received much attention. Current methods focus on two basic setups. The first considers an asymptotic regime, where $k$ is fixed the sample size $n$ is very large \cite{quesenberry1964large,goodman1964simultaneous}. The second line of work addresses a fixed $n$, where the alphabet size $k$ is relatively small. Here, the most popular SCI scheme is arguably of Sison and Glaz \cite{SisonGlaz1995}. In their work, Sison and Glaz (SG) proposed a method which utilizes Edgeworth expansions to approximate the desired distribution. Through extensive simulations, they showed that their method leads to smaller SCIs while maintaining a coverage rate closer to the desired level, compared to known methods at the time. Unfortunately, the SG scheme does not perform well in cases where the expected symbol counts are disparate \cite{MayJohnson1997Properties}. Recently, \cite{marton2022good} introduced a bootstrap framework for the case where both $k$ and $n$ are large. Yet, this approach is based on bootstrap sampling and does not provide solid theoretical guarantees. To the best of our knowledge, no method directly addresses the large alphabet regime, with provable performance guarantees. 

As above, let $\mathcal{X}$ be a countable alphabet. Here, we assume that the alphabet size $k$ is finite and known. Denote $p=p_1,\dots,p_k$ as the unknown probability distribution over $\mathcal{X}$, while $X^n$ is a collection of $n$ samples from  $\mathcal{X}$. Let $S(X^n)$ be an $(1-\alpha)$-level confidence region (CR) for $p$. That is,  
$P(p \in S(X^n))\geq 1-\alpha.$
The most popular form of a CR is the case where $S(X^n)$ is rectangular. That is, $S(X^n) = T_1(X^n), T_2(X^n), \dots, T_k(X^n)$, where $T_j(X^n)=[a_j,b_j]$ for $j=1,\dots,k$ and $0 \leq a_j\leq b_j\leq 1$. This implies SCIs for the collection of the $k$ parameters, similarly to (\ref{classical}). In fact, all the multinomial inference schemes mentioned above are rectangular CRs.

The most basic rectangular CR may be obtained from a binomial viewpoint. That is, one may construct a binomial CI for each symbol independently and correct for multiplicity using a Bonferroni correction. Hence, the obtained SCIs are just a collection of binomial CIs (of confidence level $\alpha/k$), for every symbol in the alphabet. Naturally, this approach controls the prescribed confidence level (for every $n$ and $k$), but may be over pessimistic and result in large volume SCIs. 
Notice that such an approach also applies for unobserved symbols. Specifically, a binomial CI for symbols with zero counts is obtained by a Bonferroni-corrected ROT, as described in Section \ref{Previous_Work}. 

Let us now introduce our proposed large alphabet inference scheme. We distinguish between observed and unobserved symbols. Assume that $n\leq k$ and let $\alpha_1,\alpha_2 \in [0,1]$ be two fixed constants. First, we set a binomial CI of level $1-\alpha_1/n$ for all the symbols that appear in the sample, while unobserved symbols are set a naive CI, $T_j(X^n)=[0,1]$. We have that
\begin{align}
    P(p \notin S(X^n))&=P(\cup_{j=1}^k\{p_j \notin T_j(X^n)\})\leq \sum_{j=1}^kP(p_j \notin T_j(X^n))=\\\nonumber
    &\sum_{j|N_j(X^n)=0}P(p_j \notin T_j(X^n))+\sum_{j|N_j(X^n)>0}P(p_j \notin T_j(X^n))=\\\nonumber
    &\sum_{j|N_j(X^n)>0}P(p_j \notin T_j(X^n))\leq n \cdot \frac{\alpha_1}{n}=\alpha_1
\end{align}
where the first inequality follows from the union bound and the second inequality is due to $|\{j|\;N_j(X^n)>0\}|\leq n$ (that is, the number of symbols that appear in the sample is not greater than the sample size). Notice that in the case where $n>k$, we may define binomial CI of level $1-\alpha_1/k$ for all the symbols that appear in the sample, and the above still holds. 

Now, let $A_{n}=\min_{r\geq 1}\left((q_{r,n}^*)^{r-1}(1-q_{r,n}^*)^n/\alpha_2  \right)^{1/r}$ be our proposed ($1-\alpha_2$)-level CI for unobserved events (Theorem \ref{theorem_analytical}). We would like to simultaneously control the events $M_{max}(X^n)\leq A_{n}$ and $p \in S(X^n)$ at a confidence level of $1-\alpha$. Therefore, we set $\alpha_1=\alpha(1-c)$ and $\alpha_2=\alpha c$ for some $c\in [0,1]$. We have, 
\begin{align}\label{eq2}
    &P\left(\left\{p \notin S(X^n)\}\right\}\cup\left\{M_{max}(X^n)\geq  A_n\right\}\right)\leq\\\nonumber
    &P(\cup_{j=1}^k\{p \notin S(X^n)\})+P(M_{max}(X^n)\geq A_n)\leq \alpha(1-c)+\alpha c=\alpha.
\end{align}
Notice that by simultaneously controlling both of the terms above, we may replace the naive unit intervals of the unobserved events with $[0,A_n]$. This implies the following scheme (Algorithm \ref{alg:Ours}) for constructing the desired SCIs.

\begin{algorithm}[H]
	\begin{algorithmic}[1]
		\renewcommand{\algorithmicrequire}{\textbf{Input:}}
		\renewcommand{\algorithmicensure}{\textbf{Output:}}
		\REQUIRE  A sample $X^n$, alphabet size $k$ and a confidence level $1-\alpha$.
		\STATE Set $c\in[0,1]$
		\STATE Construct a Binomial CI of level $1-\alpha(1-c)/n$ for all the symbols that appear in $X^n$ 
		\STATE Construct a CI for unobserved events (following Theorem $1$ or $2$) of level $1-\alpha c$, for all the symbols that do not appear in $X^n$
	\end{algorithmic}
	\caption{Our Proposed Large Alphabet SCI's for Multinomial Proportions}
	\label{alg:Ours}
\end{algorithm}

The scheme above introduces a simple analytical framework for constructing SCIs over large alphabets. The parameter $c$ defines an inference trade-off between observed and unobserved events. Specifically, for larger values of $k$ we expect many unobserved events which corresponds to a larger value of $c$. On the other hand, if $k$ is comparable to $n$, we would probably prefer a lower value of $c$. Therefore, choosing a reasonable value for $c$ is of a natural concern. Unfortunately, the choice of $c$ also depends on the unknown underlying distributions $p$. For example, a uniform $p$ results in fewer unobserved events than a degenerate $p$. Therefore, we cannot set a $c$ value that minimizes the SCIs uniformly, for every possible $p$. However, we  show it is possible to set a value for $c$, so that our proposed scheme  provably improves upon alternative methods.

Typically, the performance of a CR is measured by its expected volume. That is, given two CRs, we say the one outperforms the other if its expected volume is smaller, while maintaining the prescribed confidence level. However, notice that in the large alphabet regime, the volume of a CR rapidly decays with the alphabet size $k$. For example, the volume of a rectangular CR with a fixed length of $L<1$ for each of its parameters demonstrates an exponential decay, $L^k$. Therefore, we focus on the log of the volume in this regime. Specifically, in each of our following experiments we measure the average log volume, as we cannot directly assess the volume by numerical means. Further, for the same reasons, we focus on the expected log volume as our analytical figure of merit.  

\begin{theorem} \label{T3}
    Let $p$ be a probability distribution over an alphabet $\mathcal{X}$ of size $k$. Let $X^n$ be $n$ independent samples from $p$. Denote $A^{BC}_n=-\log(\alpha/k)/n$ as the Bonferroni-corrected CI for unobserved events.
    Let $A_{n,c}=\min_{r\geq 1}\left((q_{r,n}^*)^{r-1}(1-q_{r,n}^*)^n/\alpha c  \right)^{1/r}$ be our proposed CI, for a confidence level of $1-\alpha c$. Define $z_0=z_{1-{\alpha}/{2k}}$ and $z_c=z_{1-{\alpha(1-c)}/{2n}}$, where $z_a$ is the $a$ quantile of a standard normal distribution. Assume there exists $c\in[0,1]$ such that 
    \begin{enumerate}[label=(\alph*)]
        \item $k\left(1-\left(1-\frac{1}{k}\right)^n\right)(z_c-z_0)+k\left(1-\frac{1}{k}\right)^n(A_{n,c}-A^{BC}_n)\leq 0$
        \item $(z_c-z_0)+(k-1)(A_{n,c}-A^{BC}_n)\leq 0$
    \end{enumerate}
    Then, for every $p\in \Delta_k$ , the following (approximately) holds, $$\mathbb{E}\log V_c\leq \mathbb{E}\log V_0,$$ 
    where $V_c$ is the volume of our proposed CR (with a choice of $c$ that satisfies the above), $V_0$ is the volume of the Bonferroni-corrected CR and the approximation follows from Wald intervals for the Binomial proportions \cite{de1820theorie}.
\end{theorem}
Theorem \ref{T3} establishes an important property of our proposed CR. Given a sample size $n$ and an alphabet size $k$, we seek a constant $c\in[0,1]$ that satisfies the (a) and (b). This  requires a simple grid search over the unit interval. Assuming we find such $c$, then we are guaranteed that Algorithm \ref{alg:Ours} outperforms the Bonferroni-corrected CR, for every $p \in \Delta_k$. The proof of Theorem \ref{T3} is located in Appendix D. 

\subsection{Large Alphabet Experiments}
Let us now demonstrate the performance of our suggested inference scheme. We focus on two benchmark distributions which represent two extreme cases. Specifically, we study the heavy-tailed Zipf's law distribution (with $s=1.01$) and the benchmark uniform distribution. In each experiment we draw $n=1000$ samples, and evaluate the 
log-volume of different CRs (for $\alpha=0.05$), as $k$ increases. We repeat this process $1000$ times to obtain an averaged log-volume. We focus on the Bonferroni-corrected CR (denoted BC in the figures that follow), the Sison-Glaz (SG) scheme \cite{sison1995simultaneous} and our proposed method (Algorithm \ref{alg:Ours}). To configure our method, we set $c$ as the largest value within the unit interval that satisfies conditions (a) and (b), for every $k$ and $n$. We justify this choice later in this section. Figure \ref{log_volume} demonstrates the results we achieve. 
\begin{figure}[ht]
\centering
\includegraphics[width =0.47\textwidth,bb=70 180 530 600,clip]{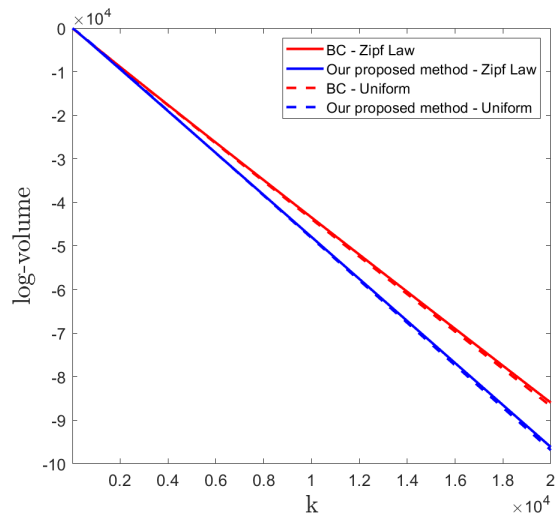}
\caption{log-volume of CR for Zipf's Law and Uniform distributions, for $n=1000$.}
\label{log_volume}
\end{figure}

First, it is evident that our proposed scheme outperforms the Bonferroni-corrected CR as $k$ grows. The SG method is omitted from Figure \ref{log_volume} as it fails to provide the prescribed confidence level (see Appendix E). It may appear from Figure \ref{log_volume} that the difference between the Zipf's Law and the uniform distribution is negligible. The reason for this phenomenon is fairly simple. For $k>>n$, most symbols do not appear in the sample (regardless to the underlying distribution). In this case, the volume of the CR is dominated by the CI of the unobserved events. This CI is fixed and independent of the sample, for both inference schemes. 
On the other hand, there is a difference in the log-volume for smaller alphabets. However, it is less visible from Figure \ref{log_volume}, and demonstrated more clearly in Figure \ref{simulated_exp_c} below. To complete the picture, we examine the coverage rate of the examined inference schemes. The results are reported in Appendix E for brevity.  As we can see, both the Bonferroni-corrected and our proposed method obtain the prescribed $0.95$ confidence level as desired, while SG fails to do so.   

Finally, we examine the performance (and sensitivity) of our suggested scheme for the choice of $c$. The upper charts of Figure \ref{simulated_exp_c} correspond to a Zipf's Law distribution ($s=1.01)$ with $k=1000$ (right) and $k=20000$ (left). The lower charts correspond to a uniform distribution with the same alphabet sizes. We use $n=1000$ samples as above. First, it is evident that for large $k$, the performance of our proposed scheme improves as $c$ grows. This is not quite surprising as there are more unobserved events in this setup. For a relatively smaller $k$, we still observe a significant improvement over the Bonferroni-corrected scheme for a large span of $c$ values, for both distributions. Despite the above, we emphasize that any choice of $c$ that satisfies conditions (a) and (b) is guaranteed to improve upon the Bonferroni-corrected scheme. Therefore, for simplicity, we choose the largest possible $c$ so that the improvement is more evident for larger alphabets.   

\begin{figure}[ht]
\centering
\includegraphics[width =0.58\textwidth,bb=30 110 550 680,clip]{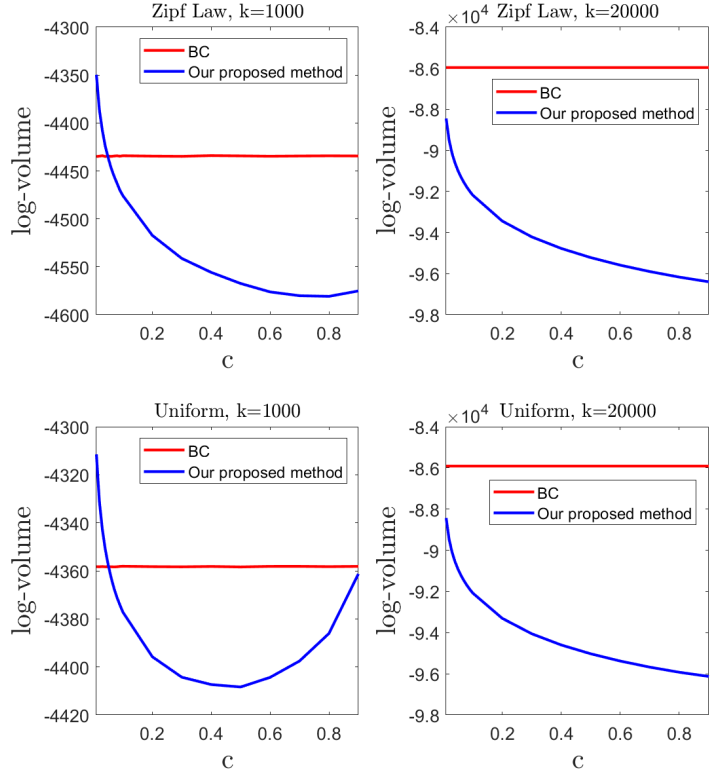}
\caption{The choice of $c$ in Zipf's Law and Uniform setups. The sample size is $n=1000$.}
\label{simulated_exp_c}
\end{figure}

\section{Discussion}
In this work we introduce an interval estimation framework for the probability of symbols that do not appear in the sample. Our suggested framework is an SoS inference scheme, designed to simultaneously control the selected parameters. We distinguish between two setups, depending on the alphabet size. First, we 
consider the case where the alphabet size $k$ is unknown and possible unbounded. This setup is of special interest in many real-world applications, as described throughout the manuscript. 
We introduce a closed-form expression for the CI, which is independent of the alphabet size. Second, we study the case where the alphabet size in known (or bounded). Here, we derive an efficient numerical routine which improves upon the unbounded $k$ solution in cases where $k$ is relatively small. It is important to emphasize that in both setups, the proposed CI is independent of the sample, similarly to the ROT. This makes it a robust framework, which is easy to apply. Next, we show that our proposed CIs are (almost) tight, in the sense that there exists a probability distribution $p$, for which we cover the missing probabilities at a confidence level (almost) equal to $1-\alpha$. We compare our proposed scheme to currently known methods, showing  significant improvement in synthetic and real-world experiments. Finally, we apply our proposed CI to large alphabet inference. Specifically, we introduce a novel scheme that provably improves upon the alternatives while controlling the desired coverage rate. 

To conclude, we revisit the motivating example in Section \ref{intro}. The $2006$ New Zealand Ministry of Health report indicated $58$ new cancer cases among Maori female under the age of $30$. Specifically, out of the $75$ studied cancer types, only $37$ were observed. Using the Bonferroni-corrected ROT, we obtain a confidence interval of $[0,0.126]$ for the unobserved cancer types. Applying our proposed scheme (under the more robust unbounded $k$ assumption), we obtain a shorter CI of length $[0,0.089]$. Similarly, for Pacific Islands men of the same age group, an additional part of the report indicated only $11$ new cancer cases in $2006$, where each case is of a different type. Here, the Bonferroni-corrected ROT suggests a CI of $[0,0.244]$ while our proposed scheme obtains a CI of $[0,0.15]$. As we can see, our interval estimation scheme demonstrates a significant improvement. This makes it an favorable alternative for this important problem, which applies to many applications.

Finally, our proposed framework may be generalized to consider the collection of probabilities with $i$ appearances in the sample. This would allow us to control more events of interest, and further improve out large alphabet inference scheme. We consider this direction for our future work.

\begin{appendices}
\section{A Proof for Property \ref{property1}}

Let $r\geq 1$ and  $h(t)=t^r(1-t)^n$. Then, the optimum of $h(t)$ satisfies
\begin{align}
    \frac{dh(t)}{dt}=t^{r-1}(1-t)^{n-1}(r(1-t)-nt)=0
\end{align}
This implies that $t^*\triangleq r/(r+n)$ is a local optimum. Further, 
\begin{align}\nonumber
    \frac{d^2h(t)}{dt^2}=&t^{r-2}(1-t)^{n-2}\bigg(r(r-1)(1-t)^2-2nrt(1-t)+n(n-1)t^2\bigg)=\\\nonumber
    &t^{r-2}(1-t)^{n-2}\bigg(t^2\big(r(r-1)+2nr+n(n-1)\big)+t\big(-2r(r-1)-2nr\big)+r(r-1)\bigg).
\end{align}
Denote the roots of the quadratic form 
$$z(t)=t^2\big(r(r-1)+2nr+n(n-1)\big)+t\big(-2r(r-1)-2nr\big)+r(r-1)$$
as $t_1$ and $t_2$. Simple calculus shows that 
$$t_{1,2}=t^*\pm \frac{1}{r+n}\sqrt{\frac{rn}{r+n-1}}.$$

As we can, $z(t)$ is quadratic and convex in $t$. This means that $z(t)<0$ for $t_1<t<t_2$ and $z(t)>0$ elsewhere. This implies that $h(t)$ is concave for $t_1<t<t_2$, and convex for $0\leq t\leq t_1$ and $t_2\leq t \leq 1$. Further, $h(t^*)<0$ which implies that $t^*$ is a local maximum. 

\section{A Proof for Theorem \ref{theorem2}}

We prove Theorem \ref{theorem2} by a series of properties. 

\begin{property}\label{property2}
Let $p^* \in \Delta_k$ be the maximizer of $E_{r,n}(p)=\sum_u h(p(u))$ where $h(p(u))\triangleq p^r(u)(1-p(u))^n$. Then, $p^*(u)=p^*(v)$ for all $p^*(u),p^*(v)\in \left[t_1,t_2\right]$.
\end{property}

\begin{proof}
By negation, assume there exists $p^*(u) \neq p^*(v)$ such that $p^*(u),p^*(v)\in \left[t_1,t_2\right]$. Define 
\begin{equation}
   \tilde{p}(l)=
    \begin{cases}
      p^*(l) & l\neq u,v\\
      \frac{p^*(u)+p^*(v)}{2} & l=u,v    \end{cases}
\end{equation}
Then, 
\begin{align}
    \sum_l h(\tilde{p}(l))=&\sum_{l \neq u,v}  h(\tilde{p}(l))+\sum_{l = u,v}  h(\tilde{p}(l))=\sum_{l \neq u,v}  h(p^*(l))+2h\left(\frac{p^*(u)+p^*(v)}{2}\right)>\\\nonumber
    &\sum_{l \neq u,v}  h(p^*(l))+h(p^*(u))+h(p^*(v))=\sum_u h(p^*(u)).
\end{align}
where the inequality follows from the concavity of $h(p(l))$ for every $p(l) \in \left[t_1,t_2\right]$. Therefore, we found $\tilde{p}\in \Delta_k$ for which $\sum_l h(\tilde{p}(l))>\sum_l h(p^*(l))$, which contradicts the optimality of $p^*$. 
\end{proof}

\begin{property}\label{property4}
Let $p^* \in \Delta_k$ be the maximizer of $E_{r,n}(p)=\sum_u h(p(u))$, where $h(p(u))\triangleq p^r(u)(1-p(u))^n$. Then, there exists at most a single $p^*(u)$ such that $p^*(u)\in \left(0,t_1\right)$.
\end{property}
\begin{proof}
By negation, assume there exist $p^*(u)$ and $p^*(v)$ such that $p^*(u),p^*(v)\in \left(0,t_1\right)$. Assume, without loss of generality, that  $p^*(v)\leq p^*(u)$. Define $\delta= p^*(v)>0$. 

Let us first assume that $p^*(u)+\delta<t_1$. The function  $h(p(u))$ is convex for $p(u) \in \left[0,t_1\right]$ and strictly convex for $p(u) \in \left[0,t_1\right)$. Therefore, we have

\begin{align}
    &h\left(p^*(u)+\delta\right)> h\left(p^*(u)\right)+\delta h'\left(p^*(u)\right)\\\nonumber
    &h\left(p^*(v)-\delta\right)\geq h\left(p^*(v)\right)-\delta h'\left(p^*(v)\right).
\end{align}
where $h'(p(u))={dh(p(u))}/{dp(u)}$. Putting together the above, we have 
\begin{align}
    h\left(p^*(u)+\delta\right)+&h\left(p^*(v)-\delta\right)>h\left(p^*(u)\right)+h\left(p^*(v)\right)+\delta\left(h'\left(p^*(u)\right)-h'\left(p^*(v)\right)\right).
\end{align}
We observe that $h'(p(u))$ is an increasing function in $p(u)$, for $p(u)\in \left(0,t_1\right)$, as its derivative, ${d^2 h(p(u))}/{dp^2(u)}$ is positive in this range. Therefore, $h'\left(p^*(u)\right)\geq h'\left(p^*(v)\right)$ and 
\begin{align}
    h\left(p^*(u)+\delta\right)+h\left(p^*(v)-\delta\right)> h\left(p^*(u)\right)+h\left(p^*(v)\right).
\end{align}
Therefore, we found $\tilde{p}\in \Delta_k$ such that 
\begin{equation}
   \tilde{p}(l)=
    \begin{cases}
      p^*(l) & l\neq u,v\\
      0 & l=v\\
      p^*(u)+\delta & l=u\\
      \end{cases}
\end{equation}
and $\sum_l h(\tilde{p}(l))> \sum_l h(p^*(l)) $, which contradicts the optimality of $p^*$.

Now, assume that $p^*(u)+\delta\geq t_1$. Then, define $\tilde{\delta}=t_1-p^*(u)>0$. We have
\begin{align}
    h\left(p^*(u)+\tilde{\delta}\right)\geq h\left(p^*(u)\right)+\tilde{\delta} h'\left(p^*(u)\right)
\end{align}
\begin{align}
    h\left(p^*(v)-\tilde{\delta}\right)> h\left(p^*(v)\right)-\tilde{\delta} h'\left(p^*(v)\right).
\end{align}
Putting together the above, we have 
\begin{align}
    h\left(p^*(u)+\tilde{\delta}\right)+&h\left(p^*(v)-\tilde{\delta}\right)>h\left(p^*(u)\right)+h\left(p^*(v)\right)+\tilde{\delta}\left(h'\left(p^*(u)\right)-h'\left(p^*(v)\right)\right).\nonumber
\end{align}
As above, we observe that $h'(p(u))$ is an increasing function in $p(u)$, for $p(u)\in \left(0,t_1\right)$. Therefore, $h'\left(p^*(u)\right)\geq h'\left(p^*(v)\right)$ and 
\begin{align}
    h\left(p^*(v)-\tilde{\delta}\right)+h\left(p^*(u)+\tilde{\delta}\right)> h\left(p^*(v)\right)+h\left(p^*(u)\right).
\end{align}
Therefore, we found $\tilde{p}\in \Delta_k$ such that 
\begin{equation}
   \tilde{p}(l)=
    \begin{cases}
      p^*(l) & l\neq u,v\\
      p^*(v)-\tilde{\delta} & l=v\\
      t_1 & l=u\\
      \end{cases}
\end{equation}
and $\sum_l h(\tilde{p}(l))> \sum_l h(p^*(l)) $, which again contradicts the optimality of $p^*$.
\end{proof}

\begin{property}\label{property3}
Let $p^* \in \Delta_k$ be the maximizer of $E_{r,n}(p)=\sum_u h(p(u))$, where $h(p(u))\triangleq p^r(u)(1-p(u))^n$. Then, there exists at most a single $p^*(u)$ such that $p^*(u)\in \left(t_2,1\right]$.
\end{property}
\begin{proof}
By negation, assume there exist $p^*(u)$ and $p^*(v)$ such that $p^*(u),p^*(v)\in \left(t_2,1\right]$. Assume, without loss of generality, that  $p^*(v)\leq p^*(u)$. Define $\delta= p^*(v)-t_2>0$. The function  $h(p(u))$ is convex for $p(u) \in \left[t_2,1\right]$ and strictly convex for $p(u) \in \left(t_2,1\right]$. Therefore, we have
\begin{align}
    h\left(t_2\right)\geq h\left(p^*(v)\right)-\delta h'\left(p^*(v)\right)
\end{align}
\begin{align}
    h\left(p^*(u)+\delta\right)> h\left(p^*(u)\right)+\delta h'\left(p^*(u)\right)
\end{align}

Putting together the above, we have 
\begin{align}
    h\left(t_2\right)+&h\left(p^*(u)+\delta\right)> h\left(p^*(v)\right)+h\left(p^*(u)\right)+\delta\left(h'\left(p^*(u)\right)-h'\left(p^*(v)\right)\right).
\end{align}
We observe that $h'(p(u))$ is an increasing function in $p(u)$, for $p(u)\in \left(t_2,1\right]$, as its derivative, ${d^2 h(p(u))}/{dp^2(u)}$ is positive in this range. Therefore, $h'\left(p^*(u)\right)\geq h'\left(p^*(v)\right)$ and 
\begin{align}
    h\left(t_2\right)+h\left(p^*(u)+\delta\right)> h\left(p^*(v)\right)+h\left(p^*(u)\right).
\end{align}
Therefore, we define $\tilde{p}\in \Delta_k$ such that 
\begin{equation}
   \tilde{p}(l)=
    \begin{cases}
      p^*(l) & l\neq u,v\\
      p^*(l)-\delta & l=v\\
      p^*(l)+\delta & l=u\\
      \end{cases}
\end{equation}
and $\sum_l h(\tilde{p}(l))> \sum_l h(p^*(l)) $, which contradicts the optimality of $p^*$.\end{proof}

\section{A Proof for Proposition \ref{prop2}}
Let $X^n$ be a sample of $n$ independent observations from $p_m \in \mathcal{U}_k$. This means that $M_{max}(X^n)={0,1/m}$, and $M_max=1/m$ if and only if there exists at least one symbol that do not appear in the sample, where all symbols are equiprobable. Therefore, the probability that $M_max=1/m$ equals the probability of placing $n$ balls in $m$ identical bins, where at least a single bin remains empty. Equivalently, 

\begin{align}\label{worst-case_appendix}
   P\left(M_{max}(X^n) = \frac{1}{m}\right)=1-\frac{m!S(n,m)}{m^n},
   \end{align}
where $m!S(n,m)$ is the number of combinations of placing $n$ distinguishable balls in $m$ distinguishable bins, where no bin is empty, and $m^n$ is the total number of combinations of placing $n$ distinguishable balls into $m$ distinguishable bins \cite{graham1989concrete}.  
\section{A proof for Theorem \ref{T3}}\label{T3_appendix}
\begin{proof}
   Let $L_i$ be the length of the CI for a symbol that appears $i$ times in the sample. The Bonferroni-corrected CR satisfies
\begin{align}
    L^{BC}_0=-\log(\alpha/k)/n\quad,\quad L^{BC}_i=2z_{1-\frac{\alpha}{2k}}\sqrt{\frac{i/n(1-i/n)}{n}}\;\; \forall i>0.
\end{align}
Notice that we use a normal approximation for the binomial CI to simplify our derivation. Next, given a fixed $c \in [0,1]$, the CI of our proposed method satisfies 
\begin{align}\nonumber
    L^{c}_0=\min_{r\geq 1}\left((q_{r,n}^*)^{r-1}(1-q_{r,n}^*)^n/(c\alpha)  \right)^{1/r}\quad,\quad L^{c}_i=2z_{1-\frac{\alpha(1-c)}{2n}}\sqrt{\frac{i/n(1-i/n)}{n}}\;\; \forall i>0,
\end{align}
where $q_{r,n}^*$ is defined in Theorem $1$. Notice that for simplicity, we use the result in Theorem $1$ although the alphabet size $k$ is known. The expected log-volume of a rectangular CR satisfies
\begin{align}
    \mathbb{E}\log V= \mathbb{E}\log \prod_{i=0}^nL_i^{\sum_u\mathbbm{1}(N_u(X^n)=i)}=\sum_{i=0}^n\sum_u\mathbb{E}\left(\mathbbm{1}(N_u(X^n)=i)\right)\log L_i.
\end{align}
We would like to find $c\in [0,1]$ such that $\mathbb{E}\log V_c\leq \mathbb{E}\log V_0$. We have,
\begin{align}
    \mathbb{E}\log V_c-\mathbb{E}\log V_0=&\sum_{i=0}^n\sum_u\mathbb{E}\left(\mathbbm{1}(N_u(X^n)=i)\right)\left(\log L_i^c-\log L_i^{BC}\right)=\\\nonumber
    &\sum_u\mathbb{E}\left(\mathbbm{1}(N_u(X^n)=0)\right)\left(\log L_0^c-\log L_0^{BC}\right)+\\\nonumber
    &\sum_{i=1}^n\sum_u\mathbb{E}\left(\mathbbm{1}(N_u(X^n)=i)\right)\left(\log L_i^c-\log L_i^{BC}\right)=\\\nonumber
    &\sum_u\mathbb{E}\left(\mathbbm{1}(N_u(X^n)=0)\right)\left(\log L_0^c-\log L_0^{BC}\right)+\\\nonumber
    &\left(\log z_{1-\frac{\alpha(1-c)}{2n}}-\log z_{1-\frac{\alpha}{2k}}\right)\sum_{i=1}^n\sum_u\mathbb{E}\left(\mathbbm{1}(N_u(X^n)=i)\right)=\\\nonumber
    &\sum_u\mathbb{E}\left(\mathbbm{1}(N_u(X^n)=0)\right)\left(\log L_0^c-\log L_0^{BC}\right)+\\\nonumber
    &\left(\log z_{1-\frac{\alpha(1-c)}{2n}}-\log z_{1-\frac{\alpha}{2k}}\right)\left(k-\sum_u\mathbb{E}\left(\mathbbm{1}(N_u(X^n)=0)\right)\right),
\end{align}
where the last equality follows from $\sum_{i=0}^n \mathbbm{1}(N_u(X^n)=i)=k$. Notice we have that $\mathbb{E}\left(\mathbbm{1}(N_u(X^n)=0)\right)=(1-p_u)^n$. Therefore, 
\begin{align}\label{eq1a}
    \mathbb{E}\log V_c-\mathbb{E}\log V_0=&\sum_u(1-p_u)^n\left(\log L_0^c-\log L_0^{BC}\right)+\\\nonumber
    &\left(k-\sum_u(1-p_u)^n\right)\left(\log z_{1-\frac{\alpha(1-c)}{2n}}-\log z_{1-\frac{\alpha}{2k}}\right).
\end{align}
Notice that (\ref{eq1a}) is linear in $\sum_u(1-p_u)^n$. Further, simple calculus shows that $\sum_u(1-p_u)^n$ attains its maximum for a uniform distribution, while its minimum is attained for a degenerate distribution. Therefore, $k(1-1/k)^n\leq \sum_u(1-p_u)^n\leq k-1$. This means that 
\begin{align}
    &\mathbb{E}\log V_c-\mathbb{E}\log V_0\leq\\\nonumber
    &\max\bigg\{  k(1-1/k)^n\left(\log L_0^c-\log L_0^{BC}\right)+k\left(1-(1-1/k)^n\right)\left(\log z_{1-\frac{\alpha(1-c)}{2n}}-\log z_{1-\frac{\alpha}{2k}}\right),\\\nonumber
    &\quad\quad\quad (k-1)\left(\log L_0^c-\log L_0^{BC}\right)+\left(\log z_{1-\frac{\alpha(1-c)}{2n}}-\log z_{1-\frac{\alpha}{2k}}\right)\bigg\}.
\end{align}
We require that $\mathbb{E}\log V_c-\mathbb{E}\log V_0 \leq0$. This holds if both arguments of the max above are non positive, as stated in conditions (a) and (b). 
\end{proof}

\newpage

\section{Coverage Rate for large Alphabet SCIs}

\begin{figure}[ht]
\centering
\includegraphics[width =0.45\textwidth,bb=50 150 530 620,clip]{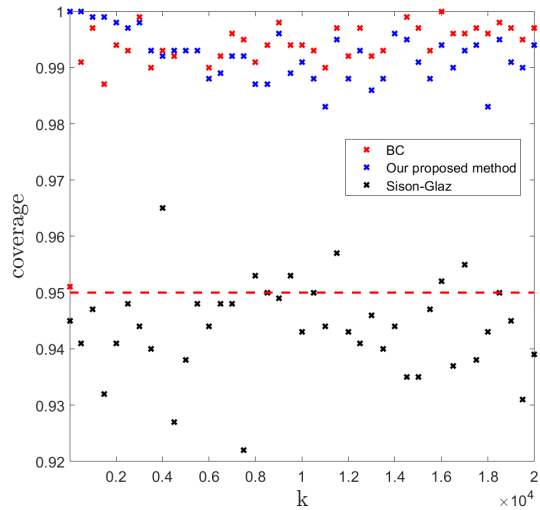}
\caption{CR coverage for a Zipf's Law distribution. The sample size is $n=1000$.}
\label{simulated_exp_coverage}
\end{figure}

\end{appendices}

\section*{Acknowledgements}
This research is supported by the Israel Science Foundation 
grant number 963/21. The author thanks Ruth Heller and Yoav Benjamini for helpful discussions.

\bibliographystyle{plain}
\bibliography{bibi}  
\end{document}